\documentclass{amsart}

\usepackage{color,graphicx,amssymb,latexsym,amsfonts,txfonts,amsmath,amsthm}
\usepackage{pdfsync}
\usepackage{amsmath,amscd}
\usepackage[all,cmtip]{xy}

\usepackage{hyperref}
\hypersetup{
    colorlinks=true,       % false: boxed links; true: colored links
    linkcolor=blue,          % color of internal links
    citecolor=blue,        % color of links to bibliography
    filecolor=blue,      % color of file links
    urlcolor=blue           % color of external links
}

\newtheorem{theo}{Theorem}[section]
\newtheorem{coro}[theo]{Corollary}
\newtheorem{prop}[theo]{Proposition}
\newtheorem{lemma}[theo]{Lemma}

\theoremstyle{remark}
\newtheorem{remark}{\bf Remark}

\def\h{{\mathbb H}}

\begin{document}

\title{Quasi-abelian group as automorphism group of Riemann surfaces}

\author{Rub\'en A. Hidalgo, Yerika Mar\'in Montilla and Sa\'ul Quispe}

\subjclass[2010]{30F10, 14H37, 14H57, 20H10, 30F50}  
\keywords{Riemann surfaces, Klein surfaces, Automorphisms, NEC groups, Dessins d'enfants}

\address{Departamento de Matem\'atica y Estad\'{\i}stica, Universidad de La Frontera. Temuco, Chile}
\email{ruben.hidalgo@ufrontera.cl}
\email{yerika.marin@ufrontera.cl}
\email{saul.quispe@ufrontera.cl}
\thanks{Partially supported by Projects FONDECYT Regular N. 1220261 and 1230001. The second author has been supported by ANID/Beca de Doctorado Nacional/ 21190335.}

\begin{abstract}
Conformal/anticonformal actions of the quasi-abelian group $QA_{n}$ of order $2^n$, for $n\geq 4$, on closed Riemann surfaces, pseudo-real Riemann surfaces and closed Klein surfaces are considered. We obtain several consequences, such as the solution of the minimum genus problem for the $QA_n$-actions, and for each of these actions, we study the topological rigidity action problem. In the case of pseudo-real surfaces, attention was typically restricted to group actions that admit anticonformal elements. In this paper we consider two cases: either $QA_n$ has anticonformal elements or only contains conformal elements.
\end{abstract}

\maketitle

%%%%%%%%%%%%%%%%%%%%
%%%%%%%%%%%%%%%%%%%%
\section{Introduction}
In \cite{Greenberg}, Greenberg proved that every finite group $G$ can be realized as a subgroup of the group ${\rm Aut}^{+}(S)$ of conformal automorphisms of a suitable closed Riemann surface $S$ (we say that $G$ provides a {conformal action} on $S$). The {\it strong symmetric genus} $\sigma^{0}(G)$ is the minimal genus of Riemann surfaces admitting a conformal action of $G$ \cite{MZ7,May3,May10}. The ones with $\sigma^0(G)=0$ correspond to the finite groups of M\"obius transformations: trivial, cyclic, dihedral and one of the Platonic symmetric groups. The description of those with $\sigma^0(G)=1$ is provided in \cite{GroT}. If $\sigma^{0}(G) \geq 2$, then it follows from the Hurwitz's bound \cite{Hurwitz} that $|G| \leq 84(\sigma^{0}(G)-1)$.
In \cite{Conder:lista}, Conder has listed the value of $\sigma^{0}(G)$ for those groups $G$ of order at most $127$. It is also known that every integer $g \geq 2$ is the strong symmetric genus for some finite group \cite{May3}. 

Let us assume we have two conformal actions of $G$ on closed Riemann surfaces $S_1$ and $S_2$ of the same genus $g$, i.e., $G \cong G_{1}<{\rm Aut}^+(S_{1})$ and $G \cong G_{2}<{\rm Aut}^+(S_{2})$. These actions are called {\it topologically equivalent} if there is an orientation preserving homeomorphism $\phi:S_{1} \to S_{2}$ such that $\phi G_{1} \phi^{-1}=G_{2}$.

There are finite groups $G$ with different topological actions in genus $\sigma^{0}(G)$. Examples of this situation can be found in Conder's list \cite{Conder:lista} (there are actions with different signatures). Note that, even if two conformal actions have the same signature, they might be topologically non-equivalent.
In \cite{HMQ22}, we have proved that, for the case of generalized quasi-dihedral group $G_n$, there is only one topological action in its strong symmetric genus.
In this paper, following similar ideas, we obtain the same rigidity result for the quasi-abelian group $QA_{n}$ of order $2^{n}$ (Corollary \ref{strong}); this action is produced on non-hyperelliptic Riemann surfaces (contrary to the case of the action of generalized quasi-dihedral group \cite[\S 3]{HMQ22}). In \cite{MZ7}, it was proved that $\sigma^{0}(QA_{n})=2^{n-2}-1$.
In Remark \ref{re1c}, we provide a description of the associated dessins d'enfants \cite{Gro} in terms of the corresponding monodromy group and the associated bipartite graph. In Theorem \ref{VJ}, we describe the isotypical decomposition \cite{AR} of the Jacobian variety $J_S$, induced by the triangular action of $QA_{n}$. 

A conformal action of $G$ is called {\it purely-non-free} if every of its element has fixed points.
In \cite{BGH}, it was observed that there is some Riemann surface on which $G$ acts purely-non-free. The {\it pure symmetric genus}  $\sigma_{p}(G)$ is the minimal genus on which $G$ acts purely-non-free. In general, $\sigma^{0}(G)$ might be much smaller that $\sigma_{p}(G)$. In \cite{HMQ22}, for the group $G_n$ we obtained that $\sigma^{0}(G_n)<\sigma_{p}(G_n)$ with $n$ odd and $\sigma^{0}(G_n)=\sigma_{p}(G_n)$ for $n$ even. In this paper, as consequence of Theorem \ref{pta} we observe that the equality also holds for the group $QA_{n}$. 

The group $QA_{n}$ has (up to automorphisms of the group) two index two subgroups $H_{1} \cong C_{2^{n-2}} \times C_{2}$ and $H_{2} \cong C_{2^{n-1}}$. In particular, this asserts that, for each $j=1,2$,  there are Riemann surfaces $S$, such that $QA_{n}<{\rm Aut}(S)$ and $H_{j}=QA_{n} \cap {\rm Aut}^{+}(S)$ (here, ${\rm Aut}(S)$ denotes the group of conformal and anticonformal automorphisms of $S$). 
As a consequence of Proulx class \cite{GroT},  the symmetric genus of $QA_{n}$ (i.e., the minimal genus of closed Riemann surfaces admitting it as a group of conformal/anticonformal automorphisms) is $\sigma(QA_{n})=1$. Let us denote by $\sigma^{hyp}(QA_{n},H_{j}) \geq 2$ the next genus as above.
We set $\sigma^{hyp}(QA_{n})={\rm Min}\{\sigma^{hyp}(QA_{n},H_{1}), \sigma^{hyp}(QA_{n},H_{2})\}$, which we call the {\it symmetric hyperbolic genus} of $QA_{n}$.
In Theorem \ref{hyp1}, we compute $\sigma^{hyp}(QA_{n},H_{j})$, from which we obtain that $\sigma^{hyp}(QA_{n})=2^{n-3}$ (Corollary \ref{hyp}). 
In Theorem \ref{tumshg}, we observe that these conformal/anticonformal actions of $QA_{n}$, in these minimal genera, are also topologically rigid. 

Another related topic that we will study in this paper is the action of finite groups on {\it pseudo-real Riemann surfaces}. Pseudo-real Riemann surfaces are those admitting anticonformal automorphisms, but none of order two. These surfaces correspond to the real points of moduli space of closed Riemann surfaces of given genus which cannot be defined over the reals. If $G$ is a finite group, admitting an index two subgroup $H$ containing all the involutions of $G$, then there is some pseudo-real Riemann surface $S$ such that $G<{\rm Aut}(S)$ and $H=G \cap {\rm Aut}^{+}(S)$. The minimal genus of such pseudo-real Riemann surfaces is denoted by $\psi^*(G)$ and called the {\em strong pseudo-real genus} of $G$. The {\em pseudo-real genus} $\psi(G)$ of $G$ is the minimum genus of pseudo-real Riemann surfaces on which $G<{\rm Aut}(S)$ but now we allow the possibility that every element of $G$ is conformal (so $\psi(G)\leq \psi^*(G)$). These parameters $\psi^{*}(G)$ and $\psi(G)$ have been studied by Conder and Lo in \cite{CL}.
In Theorem \ref{tps0}, we construct pseudo-real Riemann surfaces $S$ for which $QA_{n} \leq {\rm Aut}(S)$ and $H_{1}=QA_{n} \cap {\rm Aut}^{+}(S)$. As consequence of the Proposition \ref{QAgo}, we can observe that there are no pseudo-real Riemann surfaces of genus even with full automorphism group $QA_n$. In Theorems \ref{mps+0} and \ref{psw}, we determine that $\psi^*(QA_n)=\psi(QA_n)=2^{n-1}-1$, and in Proposition \ref{Ups0}, we observe that this minimal action is not topologically rigid (contrary to the topologically rigid of the $G_n$-action \cite{HMQ22}). 
In Theorem \ref{bpsQAn}, we observe that the order of the largest quasi-abelian group of automorphisms of a pseudo-real surfaces of genus $g$ is $2(g+1)$.

In the above, we have considered pseudo-real Riemann surfaces for which $QA_{n}$ necessarily admits anticonformal automorphisms. {Previously in \cite[Theorem 7.14]{HMQ22}, we observe that for $n$ even, there is no pseudo-real Riemann surfaces with group $G_n$ as its full group of conformal automorphisms}. Now, we may wonder for pseudo-real Riemann surfaces for which $QA_{n}$ only has conformal automorphisms. 

For $n \geq 4$, we consider the non-abelian group
$\label{gKn}K_n=\langle a, b:\ a^{8}=1, b^{2^{n-2}}=a^4, bab^{-1}=a^{\alpha}\rangle$,
with $\alpha=-1$ for $n$ even and $\alpha=3$ for $n$ odd. The group $K_n$ has order $2^{n+1}$ and it contains to the group $\langle b, a^2b^{2^{n-3}}\rangle\cong QA_n$ as an index two subgroup (both with the same number of involutions). In Theorem \ref{psrqan}, we observe that 
there are pseudo-real Riemann surfaces $S$ of genus $g=3(2^{n-1}-1)$ for $n\geq 5$ and $g=17$ for $n=4$, such that ${\rm Aut}(S)=K_{n}$ and ${\rm Aut}^{+}(S)=QA_{n}$. 

The group $K_{n}$ has two index two subgroups, one is $QA_{n}$ and the other $C_{2^{n-2}}\times C_4\cong K_{n}^{+}= \langle b^{-2}, a\rangle$ for $n\not=4$ ($C_{8}\times C_2\cong K^{+}_{4}=\langle a, b^4\rangle$ for $n=4$). The groups $K_n$, $QA_n$ and $K_n^+$ have three involutions. If $L \in \{QA_{n},K_{n}^{+}\}$, then 
let us denote by $\sigma_{ps}(K_n, L)$ the minimal genus of a pseudo-real Riemann surface $S$ such that $K_{n}<{\rm Aut}(S)$ and $L=K_{n} \cap {\rm Aut}^{+}(S)$.  In Theorem \ref{psKn}, we obtain that  $\sigma_{ps}(K_{n}, QA_{n})=3(2^{n-1}-1)$ if $n \geq 5$ (for $n=4$ such a value is $17$) and $\sigma_{ps}(K_{n}, K^{+}_{n})=2^{n}+1$. In 
particular,  $\psi^*(K_n)=2^{n}+1$ (Corollary \ref{pkn}). 

In the final section, we study the {\em symmetric crosscap number} $\tilde{\sigma}(QA_{n})$ (i.e., the minimal topological genus of closed Klein surfaces admitting it as a conformal/anticonformal action). Proceeding in a similar fashion as in \cite{HMQ22}, we obtain that $\tilde{\sigma}(QA_{n})=2^{n-2}+2$ in Theorem \ref{ncqa}, and that this minimal action is topologically rigid, dissimilar from the non-uniqueness of the action of the group $G_n$ on this minimal genera.

%%%%%%%%%%%%%%%%%%
%%%%%%%%%%%%%%%%%%
\section{Preliminaries}
%%%%%%%%%%%%%%%%%%
Throughout this paper we denote by $C_n$ the cyclic group of order $n$, by $C_m\times C_n$ the abelian group of order $mn$ and by $C_n^2$ the abelian group $C_n\times C_n$.

%%%%%%%%%%%%%%%%%%
\subsection{Quasi-abelian group}\label{Section 1}
The {\it quasi-abelian group} of order $2^n$, where $n \geq 4$ is an integer, is the non-abelian group with the following presentation
 \begin{equation}\label{g1}
QA_{n}=\langle x,y:\ x^{2^{n-1}}=y^{2}=1, \,  [x,y]=x^{ 2^{n-2}}\rangle.
\end{equation} 
 
Using the relation $yx^{k}=x^{(2^{n-2}+1)k}y$, for $k \geqslant 1$, one may check the following properties:
 
\setlength{\leftmargini}{8mm}
\begin{itemize}
\item[(a)] Every element has a unique presentation of the form $y^{j}x^{i}$, where $j$ and $i$ are integers with $ j \in\lbrace 0, 1\rbrace$ and $0\leq i < 2^{n-1}$, and these have order
$$\mid y^{j}x^{i}\mid =\dfrac{2^{n-1}}{{\rm gcd}(i, 2^{n-1})};\quad  \mid y\mid=2.$$

 \item[(b)] The index two subgroups of $QA_{n}$ are exactly the following ones
 \begin{center}$  H_1=\langle x^{2}, y \rangle \cong C_{2^{n-2}} \times C_{2},\quad H_2=\langle x\rangle \cong C_{2^{n-1}}, \quad \ H_3=\langle yx \rangle\cong  C_{2^{n-1}}.$\end{center}
 \item[(c)] The number of involutions of the groups $QA_{n}$ and $H_1$ is three and, the number of involutions of the subgroups $H_i$ isomorphic to $C_{2^{n-1}}$ is one.
 
\item[(d)]  The conjugacy classes of $QA_{n}$ are exactly $5\cdot 2^{n-3}$, with representatives given in the following table 
\begin{center}
\scalebox{0.75}{\begin{tabular}{|c|c|c|c|c|c|c|c|c|c|c|c|c|c|c|c|c|c|c|c|c|}
\hline 
Rep. & $1$ & $x$ &$x^2$ & $\cdots$ & $x^{2^{n-2}-1}$ &$x^{2^{n-2}}$ & $x^{2^{n-2}+2}$ & $\cdots$ & $x^{2^{n-1}-2}$ &  $y$&$yx$ &$yx^{2}$ & $\cdots$ & $yx^{2^{n-2}-2}$ &$yx^{2^{n-2}-1}$ \\ 
\hline 
size & 1 & 2 & 1& $\cdots$ & 2 &1  & 1 & $\cdots$ &1&  2& 2 & 2& $\cdots$ & 2 & 2\\ 
\hline 
\end{tabular}} 
\end{center}

\item[(e)] The automorphisms of the group $QA_{n}$ are given by 
\begin{center}$\psi_{u, a, v}(x)=y^ax^{u} \quad \text{and}\quad  \psi_{u, a, v}(y)=yx^{v},$\end{center} 
where $u$-odd, $a\in \{1,2\}$ and $v\in \{0,2^{n-2}\}$. The group ${\rm Aut}(QA_{n})$ has order $2^{n}$. In particular, the automorphism $\psi_{1, 1,0}$ permutes the groups $\langle x\rangle$ and $\langle yx \rangle$.
\end{itemize}

%%%%%%%%%%%%%%%%%%
\subsection{NEC groups}\label{S2}
Let us denote by $\mathcal{L} \cong {\rm PGL}_{2}({\mathbb R})$ (resp., $\mathcal{L}^{+}\cong {\rm PSL}_{2}({\mathbb R})$) the group of isometries (resp., orientation-preserving isometries) of the hyperbolic upper half-plane $\h$.
An NEC {\em group} is a discrete subgroup $\Delta$ of $\mathcal{L}$ such that the quotient space $\h/\Delta$ is a compact surface. An NEC group contained in $\mathcal{L}^{+}$ is called a {\em Fuchsian group}, and a {\em proper} NEC {\em group} otherwise. If $\Delta$ is a proper NEC group, then $\Delta^{+}=\Delta\cap \mathcal{L}^{+}$ is called its {\em canonical Fuchsian subgroup}. Note that $[\Delta:\Delta^{+}]=2$ and $\Delta^{+}$ is the unique subgroup of index 2 in $\Delta$ contained in $\mathcal{L}^{+}$.
In general, the algebraic structure of an NEC group $\Delta$ is described by the so-called {\em signature} $s(\Delta)$ \cite{Mac, Wil}:
\begin{equation}\label{E21} s(\Delta)=(h;\pm;[m_1,\cdots, m_r];\{(n_{11},\cdots, n_{1s_1}),\cdots, (n_{k1},\cdots, n_{ks_k})\}),\end{equation} 
where $h, r, k\geq 0$ and $m_i, n_{ij}>1$ (for all $i, j$) are integers. Here $h$ is the topological genus of the surface $\h/\Delta$ and $``+"$ means that $\h/\Delta$ is {\em orientable}, and $``-"$ means that $\h/\Delta$ is {\em non-orientable}. The number $k$ is the number of connected boundary components of $\h/\Delta$. The covering map $\pi:\mathbb{H}\to \mathbb{H}/\Delta$ has exactly $r$ branch values in its interior and $s_1+\cdots+s_k$ branch values on its border.  We call $m_i$ the {\em proper periods}, $n_{ij}$ the {\em periods}, and $(n_{i1},\cdots, n_{is_i})$ the {\em period-cycles} of $s(\Delta)$. We will denote by $[-]$, $(-)$ and $\{-\}$ the cases when $r=0$, $s_i=0$ and $k=0$, respectively. When there are no proper periods and there are no period-cycles in $s(\Delta)$ we say $\Delta$ is a {\em surface group}.
The signature provides a presentation of $\Delta$ \cite{Mac, Wil}:
\begin{eqnarray*}
\text{(elliptic generators)} && \beta_i\in \Delta^{+} \quad (i=1,\cdots, r); \\
\text{(reflections)} && c_{ij}\in \Delta\setminus \Delta^{+}\quad (i=1,\cdots, k;\ j=0,\cdots, s_i); \\
\text{(boundary generators)} && e_i\in \Delta^{+}\quad (i=1,\cdots, k); \\
 \text{(hyperbolic generators)} && a_i,b_i\in \Delta^{+}\quad (i=1,\cdots, h),\ \text{if $\h/\Delta$ is orientable}; \\
  \text{(glide reflections generators)} && d_i\in \Delta\setminus\Delta^{+}\quad (i=1,\cdots, h), \ \text{if $\h/\Delta$ is non-orientable};
  \end{eqnarray*}
  were these {\em generators} satisfy the {\em relations}:
   \begin{flushleft}$\beta_i^{m_i}=1\quad (i=1,\cdots, r);$\end{flushleft}
   \begin{flushleft}$e_ic_{i0}e_i^{-1}c_{is_i}=1\quad (i=1,\cdots, k);$\end{flushleft}
   \begin{flushleft}$c_{ij-1}^{2}=c_{ij}^2=(c_{ij-1}c_{ij})^{n_{ij}}=1\quad (i=1,\cdots, k;\ j=1,\cdots, s_i);$\end{flushleft}
   \begin{flushleft}$\prod_{i=1}^r\beta_i\prod_{i=1}^ke_i\prod_{i=1}^h[a_i,b_i]=1,\quad \text{if $\h/\Delta$ is orientable};$\end{flushleft}
   \begin{flushleft}$\prod_{i=1}^r \beta_i\prod_{i=1}^ke_i\prod_{i=1}^hd_i^2=1, \quad \text{if $\h/\Delta$ is non-orientable},$\end{flushleft} where $1$ denotes the identity map $id_{\h}$ in $\h$ and $[a_i,b_i]=a_ib_ia_i^{-1}b_i^{-1}$. 
   
The hyperbolic {\em area} of $\Delta$ with signature (\ref{E21})
is the {hyperbolic area} of any fundamental region for $\Delta$, and is given by \begin{equation}\label{E22} \mu (\Delta)=2\pi\Big(\eta h+k-2+\sum_{i=1}^r\big(1-\frac{1}{m_i}\big)+\frac{1}{2}\sum_{i=1}^k\sum_{j=1}^{s_i}\big(1-\frac{1}{n_{ij}}\big)\Big),\end{equation} with $\eta=2$ or 1 depending on whether or not $\h/\Delta$ is orientable. An NEC group with signature (\ref{E21}) actually exists if and only if the right-hand side of (\ref{E22}) is greater than 0.
The {\em reduced area} of an NEC group $\Delta$, denoted by $|\Delta|^*$, is given by $\mu(\Delta)/2\pi$. If $\Gamma$ is a subgroup of $\Delta$ of finite index, then the Riemann-Hurwitz formula holds $[\Delta:\Gamma]=\frac{\mu(\Gamma)}{\mu(\Delta)}.$

%%%%%%%%%%%%
\subsection{Topologically equivalent conformal/anticonformal actions}\label{Sec:signatures}
Let $S$ be a closed Riemann surface of genus $g \geq 2$. By the uniformization theorem, up to biholomorphisms, $S=\h/K$, where $K$ is a Fuchsian surface group.
We say that a finite group $G$  {\em acts} as a group of conformal (respectively, conformal/anticonformal) automorphisms of $S$ if it can be realizable as a subgroup of ${\rm Aut}^{+}(S)$ (respectively, ${\rm Aut}(S)$).
This is equivalent to the existence of a Fuchsian (respectively, an proper NEC) group  $\Delta$, containing $K$ as a normal subgroup, and of an epimorphism $\theta: \Delta \to G$ whose kernel is $K$; we say that $\theta$ provides a conformal (respectively, conformal/anticonformal) action of $G$ on $S$. 
Two conformal/anticonformal actions $\theta_1$, $\theta_2$ are {\em topologically equivalent} if there is an $\omega\in {\rm Aut}(G)$ and an $h\in {\rm Hom}^{+}(S)$ such that $\theta_2(g)=h\theta_1(\omega(g))h^{-1}$ for all $g\in G$.
This is equivalent to the existence of automorphisms $\phi \in {\rm Aut}(\Delta)$ and $\omega \in {\rm Aut}(G)$ such that
$\theta_2=\omega\circ \theta_1\circ \phi^{-1}$.

If $G$ acts conformally, then $S/G=\h/\Delta$ is a closed Riemann surface of genus $h$ and it has exactly $n$ cone points of respective cone orders $k_{1},\ldots, k_{n}$; the tuple $(h;k_{1},\ldots,k_{n})$ is also called the {\em signature} of $S/G$. If $h=0$ and $n=3$, then we talk of a {\it triangular} action.

 %%%%%%%%%%%%%%%%%%%%%%%%%
 %%%%%%%%%%%%%%%%%%%%%%%%% 
\section{Triangular conformal actions of $QA_{n}$}
In this section, we describe the triangular conformal actions of the group $QA_{n}$ (these are related to regular dessins d'enfants \cite{Gro}). 
 
\subsection{The strong and pure symmetric genus of $QA_{n}$}

\begin{theo}\label{pta}
Let $S$ be a closed Riemann surface such that ${\rm Aut}^{+}(S)=QA_{n}$, for $n \geq 4$, and $S/QA_{n}$ has triangular signature.
Then  
\setlength{\leftmargini}{8mm}
\begin{enumerate} 
\item[ (a)] $S/QA_{n}$  has signature $(0; 2,2^{n-1}, 2^{n-1})$, and $S$ is not hyperelliptic of genus $2^{n-2}-1$. 
\item[ (b)] The action of $QA_{n}$ is purely-non-free.
\item[ (c)] A (singular) model of $S$ is given by the affine algebraic curve $v^{2^{n-1}}=u^{2^{n-2}-2}(u-1)(u+1)^{2^{n-2}+1}$ with conformal automorphisms $x(u, v)=(u, \rho_{2^{n-1}}v)$ and $y(u, v)=(-u, v^{2^{n-2}+1}/u^{2^{n-3}-1}(u+1)^{2^{n-3}+1})$. 
\end{enumerate}
\end{theo}
  
\begin{proof}
We consider the quasi-abelian group $QA_{n}$ with presentation as in \eqref{g1}. Assume that $QA_{n}$ acts in a triangular way on the closed Riemann surface $S$.
\medskip

\noindent ({\bf a})  Let $\pi: S \rightarrow \mathcal{O}= S/\langle x \rangle$ be a branched regular cover map with deck group $\langle x \rangle$. As $\langle x \rangle$ is a normal subgroup of $QA_{n}$ and $y^{2}=1\in \langle x \rangle$, the automorphism $y$ induces a conformal involution $\widehat{y}$ of the orbifold $\mathcal{O}$, so it permutes the branch values of $\pi$ (i.e., the cone points of $\mathcal{O}$) and $S/QA_{n} = \mathcal{O}/ \langle \widehat{y} \rangle$. The triangular property of the action of $QA_{n}$ on $S$ (together with the Riemann-Hurwitz formula) ensures that $\mathcal{O}$ has genus zero and that its set of cone points are given by a pair of points $p_{1}, p_{2}$ (which are permuted, but not fixed, by $\widehat{y}$) and one or both of the fixed points of $\widehat{y}$.
By the uniformization theorem, we may identify $\mathcal{O}$ with the Riemann sphere $\widehat{\mathbb{C}}$. Up to post-composition of $\pi$ with a suitable M\"obius transformation, we may also assume $\widehat{y}(z) = -z,$ $p_{1}= -1$ and  $p_{2}= 1$. As the finite groups of M\"obius transformations are either cyclic, dihedral, ${\rm A_{4}}$, ${\rm A_{5}}$ or ${\rm S_{4}}$, and as $QA_{n}$ is not isomorphic to any of them, the surface $S$ cannot be of genus zero. This (together with the Riemann-Hurwitz formula) asserts that $\pm 1$ are not the only cone points of $\mathcal{O}$, at least one of the two fixed points of $\widehat{y}$ must also be a cone point (we can suppose without loss of generality that $0$ is another cone point of $\mathcal{O}$). For the point $\infty$  we have the following.
\medskip

\noindent{\bf Case $1$} ($\infty$ is a cone point of $\mathcal{O}$). The cone points of $\mathcal{O}$ are $\pm 1$, $0$ and $\infty$. As $0$ and $\infty$ are fixed points of $\widehat{y}$, each one them has points in its preimage on $S$ with $QA_{n}$-stabilizer generated by an element of the form $yx^{k}$ with $k\not=0, 2^{n-2}$ (for otherwise a point in the preimage of $0$ or $\infty$ induces a conic point of order 2 on $S/QA_{n}$ which contradicts the fact that $0$ and $\infty$ are fixed points of $\widehat{y}$ and that them are cone points of $\mathcal{O}$). Thus, the points in the preimage of 0 and $\infty$ induce conic points of order $4r$ on $S/QA_{n}$ with $r| 2^{n-3}$, and that $0$ and $\infty$ are conic points of order $2r$ in $\mathcal{O}$. 
  
 As $x^{2l}\in \langle yx^k\rangle$ with $k\not=0, 2^{n-2}$ and $l\in \{1, \cdots, 2^{n-2}-1\}$, and this does not generate $\langle x\rangle$, the $\langle x\rangle$-stabilizer of any point on $-1$ (and also about $1= \widehat{y}(-1)$) must be stabilized by a non-trivial power $x^{m}$ (we assume $m$ is a divisor of $2^{n-1}$) such that $\langle x^{m}, x^{2l} \rangle =\langle x \rangle$. From where $m$ must be prime relative to $2l$, so we can suppose without loss of generality that $m=1$. Thereby, $\mathcal{O}$ has signature $(0;2r_1,2r_2, 2^{n-1}, 2^{n-1})$ and $S/QA_{n}$ has signature $(0;4r_1, 4r_2, 2^{n-1}),$ with $r_j|2^{n-3}$. The signature $(0;4r_1, 4r_2, 2^{n-1})$ is not admissible for the action of $QA_{n}$ on $S$ (i.e., there is no an epimorphism from a Fuchsian group with signature $(0;+;[4r_1,4r_2,2^{n-1}];\{-\})$ on  $QA_{n}$).
\medskip
  
\noindent{\bf Case $2$} ($\infty$ is not a cone point of $\mathcal{O}$). In this case, the $QA_{n}$-stabilizer of any point on $\infty$ is generated by an element of the form $y$ or  $yx^{2^{n-2}}$ (i.e., does not contain a non-trivial power of $x$). So, $\infty$ induces a conic point of order 2 on $S/QA_{n}.$ Thereby, the cone points of $\mathcal{O}$ are $\pm 1$ and $0$. As $0$ is fixed point of  $\widehat{y}$, by similar arguments as in case 1, this has points in its preimage on $S$ with $QA_{n}$-stabilizer generated by an element of the form $yx^{k}$ with $k\not=0, 2^{n-2}$. Thus, the points in the preimage of 0 induce conic points of order $4r$ on $S/QA_{n}$ with $r| 2^{n-3}$, and that $0$ is a conic point of order $2r$ in $\mathcal{O}$. 

As $x^{2l}\in \langle yx^k\rangle$ with $k\not=0, 2^{n-2}$ and $l\in \{1, \cdots, 2^{n-2}-1\}$, and this does not generate $\langle x\rangle$. A similar argument as in Case 1 is used to obtain that $\mathcal{O}$ has signature $(0;2r, 2^{n-1}, 2^{n-1})$ and $S/QA_{n}$ has signature $(0;2,4r,2^{n-1}),$ with $r|2^{n-3}$. When $r\not=2^{n-3}$, the points in the preimage of 0 has  a $QA_n$-stabilizer generated by an element of the form $yx^k$, with $k\not=0, 2^{n-2}$ even, in this case the signature $(0;2, 4r, 2^{n-1})$ is not admissible for the action of $QA_{n}$ on $S$. In case $r=2^{n-3}$, the signature $(0;2, 2^{n-1}, 2^{n-1})$ is admissible for the action of $QA_{n}$ on $S$. 

 Finally, for the second part of the statement, we obtain (by the Riemann-Hurwitz formula) that $S$ has genus $2^{n-2}-1$ and we observe that $x^{2^{n-2}}$ is the unique involution in the center of $QA_{n}$. Since the subgroup $\langle x\rangle$ acts on $S$ with signature of the form $(0;2^{n-2}, 2^{n-1}, 2^{n-1})$, by \cite[Theorem 2.7]{ABCNPW} for the subgroup $\langle x^{2^{n-2}}\rangle$ of  $\langle x\rangle$ we have that the quotient surface $S/\langle x^{2^{n-2}}\rangle$ has signature of the form $(2^{n-3}-1;+;[2,2,2, 2];\{-\})$, so this not correspond to the signature of a hyperelliptic curve of genus $2^{n-2}-1$.
  
 \medskip  
 \noindent ({\bf b}) Let $\pi:S\to \widehat{\mathbb{C}}$ be a branched regular cover map with deck group $\langle x\rangle$ as in (a). Then can be seen that $x$ and $x^{2^{n-2}}$ has exactly 2 and 4 fixed points, respectively. As $(yx^{2^{n-3}})^{2}=x^{2^{n-2}}$, the automorphism $yx^{2^{n-3}}$ induce an involution $\widehat{yx^{2^{n-3}}}$ of the orbifold $\mathcal{O}$ with exactly two fixed points. On the other hand, since the point $0$ is the projection of the fixed points of $(yx^{2^{n-3}})^2$ (proof of the item (a)), and the fact that the number of elements in the conjugacy class of $yx^{2^{n-3}}$ is $2$, we have that $yx^{2^{n-3}}$ have exactly two fixed points. Finally, as $y$ and $yx^{2^{n-2}}$ are in the same conjugacy class (Section \ref{Section 1}), from the proof of the item (a) (i.e., the point $\infty$ have points in its preimage with $QA_{n}$-stabilizer generated by elements of the form $y$ or $yx^{2^{n-2}}$), we conclude that every element of $QA_{n}$ acts with fixed points.

\medskip  
 \noindent ({\bf c}) By item (a) the quotient orbifold $S/QA_{n}$ has signature $(0;2, 2^{n-1}, 2^{n-1})$, and $\pi: S \rightarrow \widehat{\mathbb{C}}$ is a cyclic branched regular covering, branched at the points $\pm 1$ (with branching order $2^{n-1}$) and at the point $0$ (with branching order $2^{n-2}$). Then by \cite{BrW}, an equation for $S$ must be of the form $w^{2^{n-1}}=t^{\alpha}(t-1)^{\beta}(t+1)^{\gamma},$ where $\alpha, \beta, \gamma \in \lbrace 1, \ldots ,2^{n-1}-1 \rbrace$ are such that: (i) $\beta$ and $\gamma$ are both relative primes to $2^{n-1}$, (ii) ${\rm gcd}(\alpha,2^{n-1})=2$, (iii) $\alpha+\beta + \gamma \equiv 0 \, {\rm mod} \, (2^{n-1})$. In this model, $\pi$ corresponds to the projection $(t, w) \mapsto t$. By the condition (ii)  and (iii) we may suppose, without loss of generality, that $\alpha=2^{n-2}-2$ and $\beta=1$; so $\gamma= 2^{n-2}+1$. In this way, we have obtained the uniqueness, except for isomorphisms of $S$.  \end{proof} 
 %%%%%%%%%%%%%%

\begin{remark}[\small{Dessins d'enfants}]\label{re1c} The bipartite graph, associated to the regular dessin d'enfant on $S$ induced by the triangular action of $QA_{n}$, is the graph $K_{2, 2}^{2^{n-2}}$, which is obtained from the complete bipartite graph (see, \cite[pp. 17]{D}) $K_{2, 2}$ in which each of its edges is replaced by $2^{n-2}$ edges. In order to observe this, we consider a regular branch cover $\beta:S\to \widehat{\mathbb{C}}$, whose deck group is $QA_{n}$. We may assume that $\beta=Q\circ P$, where $P:S\to \widehat{\mathbb{C}}$ has deck group $C_{2^{n-1}}\cong\langle x\rangle$, and $Q(x)=x^2$, whose deck group is $QA_{n}/\langle x\rangle\cong C_2$.

 Now we described the corresponding monodromy group. Let us consider the following permutations of the symmetric group $\mathfrak{S}_{2^n}$:
$$\eta=(1,2,\ldots,2^{n-1})(2^{n-1}+1, 2^{n-1}+2,\ldots, 2^n),$$
$$\sigma=\prod_{1=k=2l+t}^{2^{n-2}}(k,2^{n-1}+(2^{n-2}(1-t)+k))\prod_{2^{n-1}+1=k}^{2^{n-1}}(k,2^{n-1}+(k-2^{n-2}(1-t))).$$
Then, $\eta^{2^{n-1}}=\sigma^{2}=1$,  $\sigma \eta \sigma=\eta^{2^{n-2}+1}$, and  $\langle \eta, \sigma \rangle \cong QA_{n},$ where the isomorphism is the one taking $\eta$ to $x$ and $\sigma$ to $y$.  If $\tau=\sigma \eta^{2^{n-1}-1}$, then $\sigma \tau \eta =1$, and the pair $(\eta, \tau)$ determines the monodromy group associated to the regular dessin d'enfant of signature $(0;2, 2^{n-1}, 2^{n-1})$ as described in the Theorem \ref{pta}. This permits to see that the associated bipartite graph of this dessin d'enfant is $K_{2,2}^{2^{n-2}}$. For $n=4$, we have the following dessin d'enfant (Figura \ref{f1}).
 \end{remark}   
  \begin{figure}[h!]
 \begin{center}
\includegraphics[width=3cm]{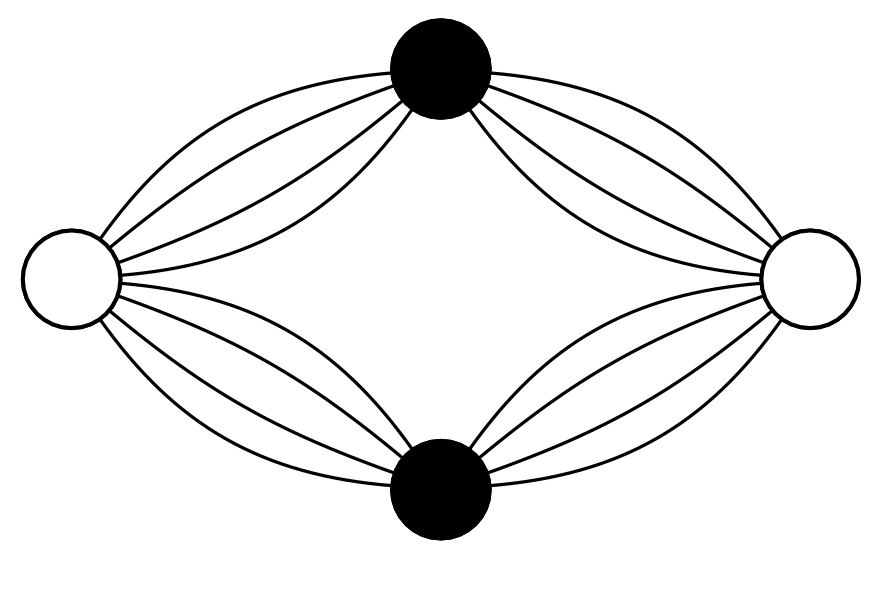} 
\caption{The graph $K_{2,2}^4$}
\label{f1}	
\end{center} 
\end{figure}
 In \cite{MZ7}, was previously proved that $\sigma^{0}(QA_{n})$ is equal to $2^{n-2}-1$.
 
\begin{coro}[Strong and pure symmetric genus of $QA_{n}$]\label{strong} 
If $n\geq 4$, then $\sigma_p(QA_{n})=2^{n-2}-1=\sigma^{0}(QA_{n})$ and, up homeomorphisms, the conformal action of $QA_{n}$ is unique.  

\end{coro}
\begin{proof}
 The group $QA_{n}$ acts purely-non-free on closed Riemann surfaces $S$ of genus $2^{n-2}-1$ with signature $(0;2,2^{n-1}, 2^{n-1})$ (Theorem \ref{pta}). As this produces the smallest possible  
hyperbolic area \cite{May10} for a conformal action of $QA_{n}$, we obtain that $\sigma_p(QA_{n})=2^{n-2}-1$.

Let us show that the conformal action of $QA_n$ is unique. Let $\Gamma$ be a Fuchsian group with signature $(0;+; [2, 2^{n-1}, 2^{n-1}];\{-\})$ and presentation 
$\Gamma=\langle \beta_1,\beta_2,\beta_3:\ \beta_1\beta_2\beta_3=\beta_1^2=\beta_2^{2^{n-1}}=\beta_3^{2^{n-1}}=1\rangle$. 
Every possible epimorphism from $\Gamma$ to $QA_{n}$, with torsion-free kernel,  is of the form
$$\theta_{k,r}:\Gamma\to QA_{n}: 
\theta_{k,r}(\beta_1)=yx^{k},\ \theta_{k,r}(\beta_2)=yx^{2^{n-2}k-(k+r)},\ \theta_{k,r}(\beta_3)=x^{r},$$
where 
$k\in\{0,2^{n-2}\}$, $r$-odd.
By post-composing $\theta_{k,r}$ by automorphisms $\psi_{u,2,0}$ and $\psi_{1,2,v}$ of $QA_{n}$,
such that $ur\equiv 1 \, {\rm mod} \, (2^{n-1})$ and $ku+v\equiv 0 \, {\rm mod} \, (2^{n-1})$, we obtain $\theta_{0,1}$. So, this action is topologically unique. 
\end{proof}

%%%%%%%%%%%%%%%%%%%%%%%%%%%%%%%%%%
\subsection{Jacobian variety for the triangular action of $QA_{n}$} 
A conformal action of a finite group $G$ on a Riemann surface $S$ of genus $g \geq  2$ induces a natural ${\mathbb Q}$-algebra homomorphism $\rho : {\mathbb Q}[G]\to {\rm End}_{\mathbb Q}(J_S)$, from the group algebra ${\mathbb Q}[G]$ into the endomorphism algebra of the Jacobian variety $J_S$. The factorization of ${\mathbb Q}[G]$ into a product of simple algebras yields a decomposition of $J_S$ into abelian subvarieties, called the {\it isotypical decomposition} \cite{LR}.
 We proceed to describe the decomposition of the Jacobian variety for the triangular action of the group $QA_{n}$ on closed Riemann surfaces.
       
\begin{theo}\label{VJ}
Let $S$ be a closed Riemann surface of genus $2^{n-2}-1$ admitting the conformal action of the group $QA_{n}$, presented as in (\ref{g1}), with signature $(0; 2,2^{n-1}, 2^{n-1})$. Then $$J_S\sim J_{S/\langle x^{2^{n-2}}\rangle}\times J_{S/\langle y\rangle}^2.$$ Moreover, $J_S$ has complex multiplication.
\end{theo}
\begin{proof} As consequence of Kani-Rosen's Theorem \cite{KR}, applied on the subgroup $H=\langle x^{2^{n-2}}, y\rangle\cong C_{2}^2$ of $QA_{n}$ with partition $H=\langle x^{2^{n-2}}\rangle \cup \langle y \rangle \cup \langle yx^{2^{n-2}} \rangle,$ and by similar arguments to the proof of the Theorem 3.3 in \cite{HMQ22}, the desired is obtained. Finally, we conclude by an algorithm in \cite{BRT} and \cite[Cor. 3.11]{FGP} that the Jacobian variety $J_S$ has complex multiplication.  \end{proof}

\begin{remark} In the proof of the second part of the item (a) of Theorem \ref{pta}, we observe that the quotient surface $S/\langle x^{2^{n-2}}\rangle$ is of genus $2^{n-3}-1$, and  since $S$ has genus $2^{n-2}-1$ we conclude that the quotient surface $S/ \langle y \rangle$ has genus $2^{n-4}$. 
\end{remark}

%%%%%%%%%%%%%%%%%%%%%%%%%%%
%%%%%%%%%%%%%%%%%%%%%%%%%%%

\section{The symmetric hyperbolic genus of $QA_{n}$}
The group quasi-abelian $QA_{n}$, presented as in (\ref{g1}), contains three index two subgroups $QA_n^+\in \{H_1\cong C_{2^{n-2}}\times C_2 , H_2\cong C_{2^{n-1}}, H_3\cong C_{2^{n-1}}\}$ (Section \ref{Section 1}). Then there are closed Riemann surfaces $S$ of genus $g\geq 2$, for which $QA_{n} \leq {\rm Aut}(S)$ and  $QA_n^+\!=\!QA_{n}\! \cap\! {\rm Aut}^{+}(S)$\! \cite{Gre}, we denote by  $\sigma^{hyp}(QA_{n},QA_n^+) \geq 2$ the smallest genus of these surfaces.

\begin{theo}\label{hyp1}
If $n \geq 4$, then 
\begin{enumerate} 
 \item $\sigma^{hyp}(QA_{n}, H_1)=2^{n-2}-1$.
 \item $\sigma^{hyp}(QA_{n}, H_i)=2^{n-3}$ with $i=2,3$.
 \end{enumerate}
\end{theo}

The symmetric hyperbolic genus $\sigma^{hyp}(QA_{n})$ of the group $QA_{n}$ is given by  $\sigma^{hyp}(QA_{n})={\rm min}\{\sigma^{hyp}(QA_{n},QA_n^+)\}$. The above provides the following.

\begin{coro}\label{hyp}
If $n \geq 4$, then $\sigma^{hyp}(QA_{n})=2^{n-3}$.
\end{coro}

%%%%%%%%%%%%%%%
\subsection{Proof of Theorem \ref{hyp1}}
Let $S$ be a closed Riemann surface of genus $g\geq 2$ such that $QA_{n} \leq {\rm Aut}(S)$ and $QA_{n}^{+}=QA_{n} \cap {\rm Aut}^{+}(S)$ is any one of index two subgroups of $QA_n$. As the orders of the cyclic subgroups of $QA_{n}$ are divisors of $2^{n-1}$, $2^{n-2}$ and $2$, the orders of the conical points (if any) of the quotient orbifold $\mathcal{O}=S/QA_{n}^{+}$, are also of that form. The group $QA_{n}$ induces an anticonformal involution $\tau$ of the orbifold $\mathcal{O}$, so that $\mathcal{O}/ \langle \tau \rangle=S/QA_{n}$, moreover $\tau$ permutes the cone points of  $\mathcal{O}$ preserving the orders and acts without fixed points if $QA_{n}^{+}=H_1$, in otherwise acts with fixed points.
\medskip
 
\noindent{\bf Case} $QA_{n}^{+}=H_1$. By above, the number of conical points (if any) of $\mathcal{O}$ is even, say $2r$, and they are permuted in pairs by the involution $\tau$. This, in particular, asserts that $\mathcal{O}$ has signature of the form  $(h ; m_{1}, m_{1}, \ldots, m_{r}, m_{r})$ (see, \cite[Corollary 2.2.5]{BuEtaGamGro}), where $m_{j}\geqslant 2$, and $S/QA_{n} = \mathcal{O}/ \langle \tau \rangle$ is a closed hyperbolic non-orientable surface, say a connected sum of $h + 1$  real projective planes with $r$ cone points of orders $m_{1}, \ldots, m_{r}$. This means that there is a proper NEC group $\Delta$ with presentation $\Delta=\langle d_{1}, \cdots , d_{h +1}, \beta_1,\cdots, \beta_r:\ \beta_1^{m_{1}}=\cdots=\beta_r^{m_{r}}=\beta_1\cdots \beta_rd_{1}^2 \cdots d_{h+1}^2=1\rangle,$ where $d_{j}$ is a glide reflection and $\beta_j$ is an elliptic transformation, and there is an epimorphism $\theta:\Delta\to QA_{n}$ such that $ \theta(\beta_j)\in QA_{n}^{+}$ and $ \theta(d_j)\in QA_{n}\setminus QA_{n}^{+}$, with torsion-free kernel $\Gamma$ and $S = \mathbb{H}/\Gamma$. 
 By \cite{May10}, the minimal genus for $S$ is when $\mathcal{O}$ has signature
$ (0;2,2,2^{n-2},2^{n-2}).$ In this case, $\Delta$ has signature $(1;-;[2, 2^{n-2}];\{-\})$ and $\theta$ is given by $ \theta(d_{1})=x; \theta(\beta_1)=y; \theta(\beta_2)=yx^{2^{n-1}-2}.$ The Riemann-Hurwitz formula asserts that $S$ has genus $2^{n-2}-1$.

\medskip
\noindent{\bf Case} $QA_{n}^{+}\not=H_1$. In this case, the number of conical points (if any) of $\mathcal{O}$ could be odd, say $2r+t$, where the $2r$ points are permuted in pairs by the involution $\tau$, and the $t$ points are associated to the fixed points of  $\tau$ (the ovals). This, asserts that $\mathcal{O}$ has signature of the form \cite[Corollary 2.2.5]{BuEtaGamGro}: $$(h ; m_{1}, m_{1}, \ldots, m_{r}, m_{r}, n_{11},\ldots, n_{1s_{1}},\cdots, n_{k 1},\ldots, n_{k s_{k}}),$$ where $m_{j}, n_{l s_{l}}\geqslant 2$, $s_1+\cdots+s_k=t$, and $S/QA_{n}= \mathcal{O}/ \langle \tau \rangle$ is a bordered surface. By \cite{May10}, a minimal genus for $S$ is obtained when $\mathcal{O}$ has triangular signature $(0;m_1,m_2, m_3)$. So, there is a Fuchsian group $\Delta^+$ with signature $(0;+; [m_{1}, m_{2}, m_{3}];\{-\})$ as canonical Fuchsian group of an NEC group $\Delta$.
Then we have three possible signatures to consider for $\Delta$ \cite{Bu0}:
$$i)\ (0;+; [-];\lbrace ( m_{1}, m_{2}, m_{3})\rbrace),\quad ii)\ (0;+; [-];\lbrace ( m_{1}, m_{1}, m_{3})\rbrace),\quad iii)\ (0;+; [m_{1}];\lbrace ( m_{2})\rbrace).$$

\medskip
\noindent{\bf Cases} $i$) {and} $ii)$. In this cases, $\Delta$ has presentation
$\Delta=\langle c_{10}, c_{1 1}, c_{1 2}, c_{1 3}, e_{1}:\ c_{1j}^{2}=(c_{1 j-1}c_{1j})^{m_{j}}=1, e_{1}^{-1}c_{1 0}e_{1}c_{1 3}= e_{1}=1\rangle,$
and there exist no an epimorphism $\theta: \Delta \to QA_{n}$. If  such an epimorphism $\theta$ exists, then it has to preserve the relations of $\Delta$, in particular $(\theta(c_{1j}))^{2}=1$ and  $\theta(e_{1})=1$. But, $\theta(c_{1j})\in \lbrace y, yx^{2^{n-2}} \rbrace $ (Section \ref{Section 1}), and these elements do not generate the group $QA_{n}$.

\medskip
\noindent {\bf Case} $iii)$. By results in \cite{Ha}, a minimal genus for $S$ is possible when $\mathcal{O}$ has signature $(0;2^{n-1}, 2^{n-1}, 2)$. So, $\Delta$ has signature $(0; +;[2^{n-1}]; \{(2)\}$ and presentation
 $\Delta=\langle \beta_1, c_{10},c_{11},e_1:\ \beta_1^{2^{n-1}}=c_{1j}^2=(c_{10}c_{11})^2=1, e_1c_{10}e_1^{-1}c_{11}=\beta_1e_1=1\rangle.$
 
 An epimorphism $\theta:\Delta\to QA_{n}$ with torsion-free kernel such that $\theta(\Delta^{+})=QA_n^+$, is given by
  $\theta(\beta_1)=x,\ \theta(e_1)=x^{-1},\ \theta(c_{10})=y,\ \theta(c_{11})=yx^{2^{n-2}}$ (for $QA_n^+=H_2$) or $\theta(\beta_1)=yx,\ \theta(e_1)=(yx)^{-1},\ \theta(c_{10})=y,\ \theta(c_{11})=yx^{2^{n-2}}$ (for $QA_n^+=H_3$). The Riemann-Hurwitz formula asserts that $S$ has genus $2^{n-3}.$
 
 In summary, the genus $\sigma^{hyp}(QA_{n})$ is equal to $2^{n-3}$, with $QA_{n}^{+}\in \{H_2, H_3\}$.  
 \qed

%%%%%%%%%%%%
\subsection{Uniqueness on the symmetric hyperbolic genus}
\begin{theo}\label{tumshg}
Let $n \geq 4$ and $QA_n^+$ be an index two subgroup of $QA_{n}$. 
Then 
the  action of $QA_{n}$ (admitting anticonformal elements) on the genus $\sigma^{hyp}(QA_{n},QA_n^+)$ is unique (up to homeomorphisms). In particular, the  action of $QA_{n}$ on the genus $\sigma^{hyp}(QA_{n})$ is unique.
\end{theo}
\begin{proof} We consider the group $QA_{n}$, and the index two subgroups $QA_n^+ \in \{H_1, H_2, H_3\}$. 
Let $S/QA_{n}=\h/\Delta$, where $\Delta$ is a proper NEC group,  and  $S/QA_n^+=\h/\Delta^{+}$, with $\Delta^{+}$
 the corresponding canonical Fuchsian group of $\Delta$. 
We need to prove that, up to pre-composition by automorphisms of $\Delta$ and post-composition by automorphisms of $QA_{n}$, there is exactly one epimorphism $\theta:\Delta\to QA_{n}$ such that $\theta(\Delta^{+})=QA_n^+$  with torsion-free kernel. 

\medskip
\noindent{\bf Case $QA_n^+= H_1$}.
In this case, as consequence of the proof of Theorem \ref{hyp1}, the quotient surface $S/QA_{n}$ has signature $(1;-;[2, 2^{n-2}];\{-\})$ and the NEC group $\Delta$ has presentation $\Delta=\langle d_1,\beta_1, \beta_2:\ \beta_1^2=\beta_2^{2^{n-2}}=1, \beta_1\beta_2d_1^2=1\rangle$, and the canonical Fuchsian group $\Delta^+$ has presentation 
$\Delta^{+}=\langle d_1^2,\beta_1,\beta_2, d_1\beta_1d_1^{-1}, d_1\beta_2d_1^{-1}\rangle.$ 
The condition for $\theta$ to have torsion-free kernel ensures that $\theta(\beta_{1})$ is an order two element of $QA_{n}$ and $\theta(\beta_{2})$ is an order $2^{n-2}$ element of $QA_{n}$, that is,
$\theta(\beta_{1}) \in \lbrace y, yx^{2^{n-2}}\}$ and $\theta(\beta_{2}) \in \lbrace yx^{j}, \, {\rm gcd}(j,2^{n-1})=2 \rbrace$.
As $\theta$ is surjective and $\theta(\Delta^{+})=QA_n^+$, we must also have $\theta(d_{1})=y^tx^{m}$, where $m$ is odd and $t\in\{0,1\}$.
As $\theta$ is homomorphism, we must also have the relation 
$$(*) \quad \theta(\beta_{1})\theta(\beta_{2})\theta(d_{1})^{2}=1.$$

If $t=0$, then (up to post-composition by automorphisms of $QA_{n}$ of the form $\psi_{u,2,0}$ (where $um\equiv 1 \, {\rm mod} \, (2^{n-1})$) and $\psi_{1,2,v}$ (where $2^{n-2}u+v\equiv 0 \, {\rm mod} \, (2^{n-1})$)), we may assume $\theta(d_{1})=x$ and $\theta(\beta_1)=y$. Now, the equality $(*)$ asserts that the value of $\theta(\beta_{2})$ must be $yx^{2^{n-1}-2}$, i.e., all of these possibilities are equivalent to
$$\theta_1:\Delta\to QA_{n}:\ \theta_1(d_1)=x;\ \theta_1(\beta_1)=y;\ \theta(\beta_2)=yx^{2^{n-1}-2}.$$ 

If $t=1$, proceeding in a similar fashion as in the previous case, we may obtain that $\theta$ is equivalent to 
$$\theta_2:\Delta\to QA_{n}:\ \theta_2(d_1)=yx;\ \theta_2(\beta_1)=y;\, \theta_{2}(\beta_2)=yx^{2^{n-2}-2}.$$

On the other hand, considering the automorphism $L$ of $\Delta$ given by $$L(d_{1})=\beta_1d_1^{-1}; \, L(\beta_1)=d_1\beta_1d_1^{-1}; \ L(\beta_2)=\beta_{2}^{-1},$$ we note that $\widehat{\theta}=\psi_{u,2,0} \circ \psi_{1,2,v}\circ \theta_1\circ L:\Delta\to QA_{n}$ is given by 
$ \widehat{\theta}=(d_1)=yx;\ \widehat{\theta}=(\beta_1)=y;\ \widehat{\theta}=(\beta_2)=yx^{2^{n-2}-2},$
where $v+2^{n-2}\equiv \, 0\, ({\rm mod}\ 2^{n-1})$ and $u(v-1)\equiv \, 1\, ({\rm mod}\ 2^{n-1})$, which corresponds to $\theta_2$. Therefore, $\theta_1$ is equivalent to $\theta_2$.  

\medskip
\noindent{\bf Case $QA_n^+\in \{H_2, H_3\}$}. 
By the proof of Theorem \ref{hyp1}, the quotient surface $S/QA_{n}$ has signature of the form $(0;+;[2^{n-1}];\{(2)\})$. So, 
 $\Delta\!=\!\langle \beta_1, c_{10}, c_{11}, e_1:\! \beta_1^{2^{n-1}}\!=\!c_{1j}^2\!=\!(c_{10}c_{11})^{2}\!=\! e_1c_{10}e_1^{-1}c_{11}\!=\!\beta_1e_1\!=\!1\rangle.$ 
In this case, 
$\theta(\beta_{1})\in \{yx^i,x^{i}, \text{$i$-odd}\}\subset QA_n^+;$ $\theta(e_1)=\theta(\beta_1)^{-1};\ \theta(c_{10}) \in \lbrace  y, yx^{2^{n-2}}\};$ $\theta(c_{11}) \in \lbrace  y, yx^{2^{n-2}}\},$ $(\theta(c_{10})\theta(c_{11}))^{2}=1$ and $\theta(e_1)\theta(c_{10})\theta(e_1)^{-1}\theta(c_{11})=1.$

All of these possibilities are equivalent (up to post-composition by some automorphisms of $QA_n$ of the form $\psi_{u,2,0}$, $\psi_{1,2,v}$ and  $\psi_{1,1,0}$) to
$$\theta_1:\Delta\to QA_{n}:\ \theta_1(\beta_1)=x;\ \theta_1(e_1)=x^{-1};\ \theta_1(c_{10})=y;\ \theta_1(c_{11})=yx^{2^{n-2}}.$$
\end{proof}

%%%%%%%%%%%%%%%%%%%%%%%%
%%%%%%%%%%%%%%%%%%%%%%%%
\section{Pseudo-real actions of the group $QA_{n}$}
In this section, we consider conformal/anticonformal actions of the group quasi-abelian $QA_{n}=\langle x, y\rangle$, presented as in (\ref{g1}), on pseudo-real Riemann surfaces. There are two cases to consider: either $QA_{n}$ has anticonformal elements or $QA_{n}$ only contains conformal elements.

\subsection{Conformal/anticonformal actions of $QA_{n}$ on pseudo-real Riemann surfaces}
In this subsection, we look for pseudo-real Riemann surfaces $S$ with $QA_{n} \leq {\rm Aut}(S)$ and $QA_{n} \neq QA_{n} \cap {\rm Aut}^{+}(S)$. As the only index two subgroup of $QA_{n}$ containing all the involutions of $QA_{n}$ is $H_1= \langle x^{2},y\rangle\cong C_{2^{n-2}}\times C_2$, we construct pseudo-real Riemann surfaces $S$ with $QA_{n} \leq {\rm Aut}(S)$ and $QA_{n} \cap {\rm Aut}^{+}(S)=H_1$.

\begin{theo}\label{tps0}
Let $l\geq 2$, $n\geq 4$ be integers and $r\geq 1$ be an odd integer. Then there are pseudo-real Riemann surfaces $S$ of genus $2^{n-2}l+2r(2^{n-2}-1)-2^{n-1}+1$, such that ${\rm Aut}(S) =QA_n$ and ${\rm Aut}^{+}(S) =H_1$.
\end{theo} 

\begin{proof}
Let $n\geq 4$ be an integer and let us consider an NEC group $\Delta$ with signature $(1;-;[2,\stackrel{l}{\cdots},2,2^{n-2},\stackrel{r}{\cdots},2^{n-2}];\{-\})$ and presentation $\Delta\!=\!\langle d_1, a_1,\! \cdots\!, a_l, b_1,\!\cdots,\!b_r\rangle$, where its generators satisfy the relations $a_1^2=\!\cdots=\!a_l^2=\!b_1^{2^{n-2}}\!=\!\cdots=\!b_r^{2^{n-2}}\!=\!\prod_{j=1}^la_j\prod_{i=1}^rb_id_1^2\!=1$, where $l\geq 2$ be an integer and $r\geq 1$ be an odd integer. 
The quotient Klein surface $\h/\Delta$ is an orbifold whose underlying surface is the real projective plane with $l+r$ conical points. 

Now, we consider the epimorphism $\theta:\Delta\to QA_n$ given by

 \noindent{($l$ even,\ $1\leq j\leq l$, $2\leq s\leq r-1$ and ${\rm gcd}(t_s, 2^{n-1})=2$)}\begin{small}
 $$\theta(d_1)=yx; \theta(a_j)=y; \theta(b_1)=x^{2^{n-2}-2}; \theta(b_{s})=x^{t_s};\theta({b_{s+1}})=x^{-t_s}.$$ 
  \end{small}
 \noindent{($l$ odd,\ $1\leq j\leq l-1$, $2\leq s\leq r-1$ and ${\rm gcd}(t_s, 2^{n-1})=2$)} 
 \begin{small}
 $$\theta(d_1)=x; \theta(a_l)=x^{2^{n-2}}; \theta(a_j)=y; 
 \theta(b_1)=x^{2^{n-2}-2}; \theta(b_{s})=x^{t_s}; \theta({b_{s+1}})=x^{-t_s}.$$
  \end{small}
  The kernel $\Gamma$ of $\theta$ is a torsion-free subgroup (contained in the half-orientation part $\Delta^{+}$ of $\Delta$) such that $S=\mathbb{H}/\Gamma$ is a closed Riemann surface with $QA_{n} \leq {\rm Aut}(S)$ and $H_1\leq {\rm Aut}^{+}(S)$. As the signature of $S/H_1$  is $(0; 2,\stackrel{2l}{\cdots},2, 2^{n-2},\stackrel{2r}{\cdots},2^{n-2})$, where the number of cone points is exactly $2l+2r\geq 6$, it follows (from Singerman list of maximal Fuchsian groups \cite{S.}) that we may choose $\Delta$ so that ${\rm Aut}(S)=QA_{n}$. In this case, as the only anticonformal  automorphisms of $S$ are the elements of $QA_{n} \setminus H_1$ (which have order different than two) it follows that $S$ is a pseudo-real Riemann surface. The Riemann-Hurwitz formula asserts that $S$ has genus $2^{n-2}l+2(2^{n-2}-1)r-2^{n-1}+1.$\end{proof}

  \begin{prop}\label{Notps}
There are not exists a faithful essential action of the group quasi-abelian $QA_{n}$, presented as in (\ref{g1}),  on a pseudo-real surface with signature $(1;-; [ j, k]; \lbrace -\rbrace)$ for $2 \leqslant j \leqslant k$, or $(2; -; [k]; \lbrace -\rbrace)$ for $k \geqslant 2$, or $(3;-;[-];\lbrace -\rbrace)$.
\end{prop}

\begin{proof}
Suppose to the contrary that the quasi-abelian group $QA_{n}$ has a faithful essential action on a pseudo-real surface $S$ with one of the given signatures, i.e., there is an epimorphism $\theta:\Delta\to QA_{n}$ with kernel a torsion-free Fuschsian group uniformizing $S$. In the first case we have an NEC group $\Delta$ with presentation $\Delta=\langle d_1,\beta_1,\beta_2:\ \beta_1^j=\beta_2^k=d_1^{2}\beta_1\beta_2=1\rangle$, and $QA_{n}$ can be generated by elements $\theta(d_1)$ and $\theta(\beta_1)$ such that $\theta(\beta_1)$ has order $j$ and $\theta(d_1)^{2}\theta(\beta_1)$ has order $k$, in the second case $\Delta$ has presentation $\Delta=\langle d_1, d_2,\beta_1:\ \beta_1^k=d_2^2d_1^{2}\beta_1=1\rangle$, and $QA_{n}$ can be generated by elements $\theta(d_1)$ and $\theta(d_2)$ such that $\theta(d_2)^{2}\theta(d_1)^2$ has order $k$. In both cases, however, since $QA_{n}$ is a split metacyclic group it admits an automorphism that inverts each of the elements in the generating pair \cite[Proposition 5.1]{BuCC20}, so the action of $QA_n$ on $S$ can be extended to one of a larger group on $S$ with reflections, contradiction.

 In third case, we have $\Delta$ with presentation $\Delta=\langle d_1,d_2,d_3:\ d_1^{2}d_2^{2}d_3^{2}=1\rangle$, where $\theta(d_j)\in QA_n\setminus \langle x^{2},y\rangle=\{y^{i}x^{l},\ i\in \{0,1\},\ l-{\rm odd} \}$, and as $\theta$ is a homomorphism, we must also have the relation $\theta(d_1)^{2}\theta(d_2)^{2}\theta(d_3)^{2}=1,$ and this implies that $x^{s}=1$ with $s\not\equiv 0\ {\rm mod}(2^{n-1})$, contradiction.
\end{proof}

Next on, we look for the genus $\psi^*(QA_n)$ and $\psi(QA_n)$.
 
 \begin{theo}\label{mps+0} If $n \geq 4$, then $\psi^*(QA_n)=2^{n-1}-1$. Furthermore, these pseudo-real Riemann surfaces are not hyperelliptic.
\end{theo}

\begin{proof} 
By Theorem \ref{tps0} (for $l=2$ and $r=1$), there exists a pseudo-real Riemann surface $S$ of genus $g_0=2^{n-1}-1$ such that ${\rm Aut}(S)=QA_{n}$, ${\rm Aut}^{+}(S)=H_1$ and $S/QA_{n}$ has signature $(1;-;[2,2,2^{n-2}];\{-\})$. 
It remains to show that $g_{0}$ is minimal with that property. By the contrary, assume there is some pseudo-real Riemann surface $R$ of genus $g<g_{0}$ such that $QA_n={\rm Aut}(R)$ and  $H_1={\rm Aut}^{+}(R)$. In this case, there is an NEC group $\Delta$ of signature $(\gamma;-;[m_1,\cdots, m_r];\{-\})$, with $\gamma>0$ and $g=1+\frac{|QA_n|}{2}|\Delta|^*$, where $|\Delta|^*=\gamma-2+\sum_{i=1}^r\left(1-\frac{1}{m_i}\right)$, such that ${\mathbb H}/\Delta=R/QA_{n}$, and there is an epimorphism $\theta:\Delta \to QA_{n}$ with kernel a torsion-free Fuchsian group uniformizing $R$. The condition $g<g_{0}$ asserts that  $|\Delta|^*<1-\frac{1}{2^{n-2}}$. Therefore, $\gamma \in \{1, 2\}$.

\noindent({\bf 1}) If $\gamma=2$, then $|\Delta|^*=\sum_{i=1}^r\left(1-\frac{1}{m_i}\right)<1-\frac{1}{2^{n-2}}$, so $r\in \{0,1\}$. The case $r=0$ is not possible by that $QA_n$ acts on Riemann surfaces of genus $g\geq 2$. If $r=1$, then $\Delta$ has signature $(2;-;[m];\{-\})$, where $m\in \{2,\cdots, 2^{n-3}\}$, by Proposition \ref{Notps} this is not possible.

 \noindent({\bf 2}) If $\gamma=1$, then $|\Delta|^*=-1+\sum_{i=1}^r\left(1-\frac{1}{m_i}\right)<1-\frac{1}{2^{n-2}}$, so $r\in \{0, 1, 2, 3\}$. The cases $r\in \{0,1\}$ are not possible by that $QA_n$ acts on Riemann surfaces of genus $g\geq 2$. If $r=2$, then $\Delta$ has signature $(1;-;[m_1, m_2];\{-\})$ with $2\leq m_1, m_2\leq 2^{n-2}$, by Proposition \ref{Notps} this is not possible. If $r=3$, then $\Delta$ has signature $(1;-;[2,2,m_3];\{-\})$, where $m_3\in \{2,\cdots, 2^{n-3}\}$,  and its canonical Fuchsian subgroup $\Delta^+$ has signature $(0;+;[2,2,2,2,m_3, m_3];\{-\})$, so ${\rm Aut}^+(R)=H_1$ acts on $R$ with signature $(0;+;[2,2,2,2,m_3, m_3];\{-\})$. In this case, by \cite[Theorem 9.1]{Br00}
 we have that $m_3=2^{n-2}$\ (contradiction to $m_3\in \{2,\cdots, 2^{n-3}\}$).
 
 For the second part of the statement, first we observe that $x^{2^{n-2}}$ is the unique involution in the center of ${\rm Aut}^+(S)$. Then, by \cite[Theorem 2.7]{ABCNPW} for the normal subgroup $\langle x^{2^{n-2}}\rangle$ of  ${\rm Aut}^+(S)=H_1$ we have that the quotient surface $S/\langle x^{2^{n-2}}\rangle$ has signature $(2^{n-2}-1;2,2,2, 2)$, and not correspond to the signature of a hyperelliptic curve of genus $2^{n-1}-1$.
 \end{proof}

 \begin{prop}\label{Ups0} The action of $QA_{n}$ as a group of conformal/anticonformal automorphisms on pseudo-real Riemann surfaces of genus $\psi^*(QA_n)$ is not topologically unique.
\end{prop}
\begin{proof} By Theorem \ref{mps+0}, we know that there exists a pseudo-real Riemann surface $S$ of genus $\psi^*(QA_n)$ such that the quotient orbifold $S/QA_{n}=\h/\Delta$ has signature of the form $(1;-;[2, 2, 2^{n-2}];\{-\})$, where $\Delta$ is an NEC group with presentation $\Delta=\langle d_1,\beta_1, \beta_2,\beta_3:\ \beta_1^2=\beta_2^2=\beta_3^{2^{n-2}}=1, \beta_1\beta_2\beta_3d_1^2=1\rangle$ and $\Delta^+$ its canonical Fuchsian subgroup with presentation 
$\Delta^{+}=\langle d_1^2,\beta_1,\beta_2, \beta_3, d_1\beta_1d_1^{-1}, d_1\beta_2d_1^{-1}, d_1\beta_3d_1^{-1}\rangle,$ and there is an epimorphism $\theta:\Delta\to QA_n$ with kernel a torsion-free Fuchsian group uniformizing $S$.  
The condition for $\theta$ to have torsion-free kernel ensures that $\theta(\beta_{1})$ and $\theta(\beta_2)$ are order two elements of $QA_{n}$ and $\theta(\beta_{3})$ is an order $2^{n-2}$ element of $QA_{n}$, that is,
$\theta(\beta_{1}), \theta(\beta_{2}) \in \lbrace y, yx^{2^{n-2}}, x^{2^{n-2}}\}$ and $\theta(\beta_{3}) \in \lbrace y^kx^{j}, \, {\rm gcd}(j,2^{n-1})=2 , k\in\{0,1\}\rbrace$.
As $\theta$ is surjective and $\theta(\Delta^{+})=H_1$, we must also have $\theta(d_{1})=y^tx^{m}$, where $m$ is odd and $t\in\{0,1\}$.
As $\theta$ is homomorphism, we must also have the relation $\ \theta(\beta_{1})\theta(\beta_{2})\theta(\beta_{3})\theta(d_{1})^{2}=1.$

By post-composing $\theta$ by automorphisms of $QA_{n}$ of the form $\psi_{u,2,0}$ and $\psi_{1,2,v}$, we obtain that all  of these possibilities for $\theta$ are ${\rm Aut}(QA_n)$-equivalent to
\begin{center}$\theta_1:\Delta\to QA_{n}:\ \theta_1(d_1)=yx;\ \theta_1(\beta_1)=y;\ \theta_1(\beta_2)=y; \theta_{1}(\beta_3)=x^{2^{n-2}-2}.$\end{center}
\begin{center}$\theta_2:\Delta\to QA_{n}:\ \theta_2(d_1)=x;\ \theta_2(\beta_1)=y;\ \theta_2(\beta_2)=yx^{2^{n-2}};\ \theta_{2}(\beta_3)=x^{2^{n-2}-2}.$\end{center}
\begin{center}$\theta_3:\Delta\to QA_{n}:\ \theta_3(d_1)=x;\ \theta_3(\beta_1)=y;\ \theta_3(\beta_2)=x^{2^{n-2}};\ \theta_{3}(\beta_3)=yx^{2^{n-2}-2}.$\end{center}

On the other hand, considering the automorphism $Q$ of $\Delta$ given by $$Q(d_{1})=\beta_2d_1; \, Q(\beta_1)=d_1^{-1}\beta_2d_1; \ Q(\beta_2)=\beta_{1};\ Q(\beta_3)=\beta_2\beta_3\beta_2,$$ we note that $\widehat{\theta}= \psi_{1,2,v}\circ \theta_1\circ Q:\Delta\to QA_{n}$ is given by 
$ \widehat{\theta}=(d_1)=x;\ \widehat{\theta}(\beta_1)=y;\ \widehat{\theta}(\beta_2)=yx^{2^{n-2}};\ \widehat{\theta}(\beta_3)=x^{2^{n-2}-2},$
 which corresponds to $\theta_2$. Therefore, $\theta_1$ is equivalent to $\theta_2$.  
 
 Accordingly, we find there are just two possibilities $\theta_1$ and $\theta_3$ for the epimorphism $\theta$.
\end{proof}

%%%%%%%%

\begin{theo}\label{psw} Let $n\geq 4$. Then $\psi(QA_n)=\psi^*(QA_n)$.
\end{theo}
\begin{proof}
Suppose to the contrary that $\psi(QA_n)<\psi^*(QA_n)$. Then $QA_n$ must be a subgroup of some larger group $L$ with $\psi^*(L)<\psi^*(QA_n)$. Let $l=\psi^*(L)$, and let $g=\psi^*(QA_n)$, which $$g=2^{n-1}-1=\frac{|QA_n|}{2}-1.$$ Note that $|QA_n|=2(g+1)>2g.$ Also, $l<g$, and so $QA_n$ cannot act essentially on a pseudo-real Riemann surface of genus $l$, and therefore $QA_n$ must be a subgroup of $L^+$, so the index $[L:QA_n]$ is even. Next, we have
$$2g[L:QA_n]<|QA_n|[L:QA_n]=|L|\leq 12(l-1)<12l=12\psi^*(L)<12\psi^*(QA_n)=12g,$$ which implies that $[L:QA_n]<6$, and so $[L:QA_n]=2$ or 4.

Now suppose that $[L:QA_n]=4$. Then $|L|=4|QA_n|>8g>8l,$ which implies that $|L|=12(l-1)$, since 8 is the second largest possible value of the ratio $R=\frac{|L|}{l-1}$. In particular $|L|=4|QA_n|$ is divisible by 4.
 
 On the other hand, because $|L|=12(l-1)$ we know that $L$ acts with signature of the form  $(1;-;[2,3];\{-\})$, and so $L$ is a quotient of some NEC group $\Gamma$ with presentation $\langle d,\beta:\ \beta^2=(d^2\beta)^3=1\rangle$.  Then since $QA_n$ is a subgroup of index 2 in $L^+$, also $QA_n$ is a quotient of an index 2 subgroup of the subgroup $\Gamma^+=\langle d^2, \beta, d^{-1}\beta d\rangle$. But there is only one subgroup $J$ of index 2 in $\Gamma^+$ generated by the elements $d^2\beta, d\beta d(=d^2(d^{-1}\beta d))$ and $\beta d^2$, all of which have order 3 (see, \cite{CL}). In particular, every non-trivial quotient of $J$ has exponent 3, and hence the order of its quotient $QA_n$ is divisible by 3 (a contradiction). 
 
 Now suppose that $[L:QA_n]=2$, this implies that $QA_n=L^+$, and $|L|=2|QA_n|$. Then
 $$|L|=2|QA_n|=4(g+1)>4g>4l>4(l-1),$$ and it follows that $L$ acts with signature $(1;-;[2,k];\{-\})$ for some $k\geq 3$, or $(1;-;[3,k];\{-\})$ where $3\leq k\leq 5$.
 
 If $L$ acts with signature $(1;-;[2,k];\{-\})$, then $|L|=\frac{4k}{k-2}(l-1)$, and $QA_n=L^+$ is a quotient of the index 2 subgroup $\Delta^+=\langle d^2, \beta, d^{-1}\beta d\rangle$ in some NEC group $\Delta=\langle d, \beta:\ \beta^2=(d^2\beta)^k=1\rangle$. Since $QA_n$ is generated by the images of the elements $\beta$, $d^2\beta$ and $d^{-1}\beta d$, which have orders 2, $k$ and 2, we see that $|QA_n|=2k$. But then
 $$4k=2|QA_n|=|L|=\frac{4k}{k-2}(l-1),$$ so $k=l+1,$ which gives $2(g+1)=|QA_n|=2k=2l+2$, so $g=l$ (a contradiction).
 
 On the other hand, if $L$ acts with signature $(1;-;[3, k];\{-\})$, then $|L|=\frac{6k}{2k-3}(l-1)$, and $QA_n=L^+$ is a quotient of the index 2 subgroup $\Delta^+=\langle d^2,\beta, d^{-1}\beta d\rangle$ with signature $(0;+;[3,3,k, k];\{-\})$, in an NEC group $\Delta=\langle d, \beta:\ \beta^3=(d^2\beta)^k=1\rangle$, where $3\leq k\leq 5$, but the signature $(0;+;[3,3,k, k];\{-\})$ is not admissible for the action of the group $QA_n$ (this is generated by elements of order 2 and $2^{n-1}$).   
 \end{proof}
 
 In the following result, we observe that there are not exists pseudo-real surface of even genus such that  ${\rm Aut}(S)=QA_n$ and ${\rm Aut}^+(S)=H_1$.
 
 \begin{prop}\label{QAgo} If $QA_n$ acts on a pseudo-real surface of genus $g\geq 2$, then $g$ is odd. 
 \end{prop}
\begin{proof} Let $n\geq 4$. Suppose that $QA_n$ acts on a pseudo-real surface $S$ of genus $g\geq 2$ such that
 ${\rm Aut}^+(S)=H_1$. As the orders of the cyclic subgroups of $H_1$ are divisors of $2^{n-2}$ and $2$, there is an NEC group $\Delta$ of signature $(\gamma;-;[2,\stackrel{l_2}{\cdots},2, 4,\stackrel{l_4}{\cdots},4,\cdots,2^{n-3},\stackrel{l_n}{\cdots},2^{n-3},2^{n-2},\stackrel{l_{2n}}{\cdots},2^{n-2}];\{-\})$ with $\gamma>0$ such that $\mathbb{H}/\Delta=S/QA_n$, and an epimorphism $\theta:\Delta\to QA_n$ with kernel a torsion-free Fuchsian group uniformizing $S$. Then, by the Riemann-Hurwitz formula we obtain $g\!=\!1+\frac{|QA_n|}{2}|\Delta|^*=1+2^{n-1}|\Delta|^*,$ where $|\Delta|^*=\frac{A}{2^{n-2}}\geq \frac{2^{n-2}-1}{2^{n-2}}$, with $A\in \mathbb{N}$. %\ (contradiction to $g$ even).
\end{proof}
 
 In \cite[Theorem 3.5]{BuCC20}, its was observe that for every odd integer $g\geq 3$, the largest order of an orientation-reversing automorphism of a pseudo-real surface of genus $g$ is $2g-2$. We have the following result.
 
 \begin{prop}\label{psLQn}
If the quasi-abelian group $QA_{n}$ has a faithful essential action on a pseudo-real surface $S$ of genus $g$ such that $|QA_{n}|> 2(g - 1)$, then $QA_{n}$ must act on $S$ with signature $(1; -; [2, 2, 2^{n-2}]; \lbrace -\rbrace)$. 
\end{prop}

\begin{proof}
Suppose that $|QA_{n}|> 2(g - 1)$ (so that $\frac{2}{R} < 1$). Then the signature of $S/QA_n$ must be either $(1;-;[j,k];\lbrace- \rbrace)$ where $2 \leqslant j \leqslant k$, or $(2;-;[k];\lbrace- \rbrace)$ for some $k \geqslant 2$, or $(1; -; [2, j, k]; \lbrace - \rbrace)$ where $2 \leqslant j \leqslant k$, but Proposition \ref{Notps} eliminates the first two of these possibilities, while in the third case $1\!>\!\frac{2}{R}\!=\!\frac{3}{2}\!-\!\frac{1}{j}\!-\!\frac{1}{k}$, so $\frac{1}{j}\!+\!\frac{1}{k}\!>\!\frac{1}{2}$, thereby $j=2$, $k=2^{n-2}$.
\end{proof}

\begin{theo}\label{bpsQAn}
The largest order of a quasi-abelian group $QA_n$ of automorphisms acting essentially on a pseudo-real surface of odd genus $g$ is $M_{QA_n}(g)$ equal to  $2g+2$.
\end{theo}
\begin{proof}
First, as consequence from the Theorem \ref{mps+0}, we have $M_{QA_n}(g)\geq 2(g+1)$ for every odd integer $g$. Next, we show that bound of $2(g+1)$ is sharp.

Let $QA_n$ be a quasi-abelian group of automorphisms of a pseudo-real surface of genus $g$, of order $|QA_n|\geq 2(g+1)$, and acting with signature $(\gamma;-;[m_1,\cdots,m_r];\{-\})$. 

Then since that $|QA_n|\geq 2(g+1)>2(g-1)$, by Proposition \ref{psLQn}, we know that $\gamma=1$ and $r=3$, with $\{m_1,m_2,m_3\}=\{2,2,2^{n-2}\}$. Then substituting into the Riemann-Hurwitz formula gives $g=2^{n-1}-1$. Accordingly, we cannot do better that $|QA_n|=2(g+1)$.
\end{proof}

 \subsection{Conformal actions of $QA_{n}$ on pseudo-real Riemann surfaces}\label{Scaqan}
In this section, we study pseudo-real Riemann surfaces with $QA_{n}$ as its group of conformal automorphisms.

For $n\geq 4$ be an integer, we consider the following non-abelian group of order $2^{n+1}$
\begin{equation}\label{gKn}K_n=\langle a, b:\ a^{8}=1, b^{2^{n-2}}=a^4, bab^{-1}=a^{\alpha}\rangle\end{equation} with $\alpha=-1$ for $n$ even and $\alpha=3$ for $n$ odd. The group $K_n$ contains to the group $\langle b, a^2b^{2^{n-3}}\rangle\cong QA_n$ as an index two subgroup, and the number of involutions of the groups $K_n$ and $QA_n$  is three. 

\begin{theo}\label{psrqan}
Let $n\geq 4$. Then there are pseudo-real Riemann surfaces $S$ of genus $3(2^{n-1}-1)$ for $n\geq 5$ and $17$ for $n=4$, such that ${\rm Aut}(S)=K_n$ and ${\rm Aut}^+(S)=QA_n$. 
\end{theo}

\begin{proof}
 Let $n\geq 4$ be an integer and let us consider an NEC group $\Delta$ with signature $(2;-;[2,t];\{-\})$ where $t=2^{n-2}$ for $n\not=4$ ($t=2$ for $n=4$) and presentation of the form
 $\Delta=\langle d_1, d_2,\beta_1,\beta_2:\ \beta_1^2=\beta_2^{t}=1, \beta_1\beta_2d_2^2d_1^2=1\rangle.$ Then quotient Klein surface uniformized by $\Delta$ is the orbifold whose underling surface is the connected sum of two real projective planes with 2 conical points. Let us consider the epimorphism $\theta:\Delta\to K_{n}$ given by $$\theta(\beta_1)\!=\!\begin{cases}
 a^2b^{2^{n-3}}, & \text{for $n\not=4$}\\
 b^4,&  \text{for $n=4$}
 \end{cases},\
 \theta(\beta_2)\!=\!\begin{cases}
b^{2^{n-2}-{2^{n-3}}-2},& \text{for $n\not=4$ even}\\
a^2b^6, & \text{for $n=4$}\\
b^{2^{n-1}-{2^{n-3}}-2},& \text{for $n$ odd} \end{cases},\ \theta(d_1)=a,\ \theta(d_2)=ab.
$$ 
  Proceeding in a similar fashion as in the proof of Theorem \ref{tps0}, we may obtain pseudo-real Riemann surfaces $S$ of genus $3(2^{n-1}-1)$ for $n\not=4$ ($17$ for $n=4$) with ${\rm Aut}(S)=K_n$ and ${\rm Aut}^+(S)=QA_{n}$. 
\end{proof}

\begin{remark}
Let $n\geq 5$. Then the action above of $K_n$ with $QA_{n} \leq {\rm Aut}^{+}(S)$ is not unique on pseudo-real Riemann surfaces of genus $3(2^{n-1}-1)$. In fact, there exists an NEC group $\Delta$ with signature $(1;-;[2,2^{n-1}, 2^{n-1}];\{-\})$ and an epimorphism $\theta:\Delta\to K_n$  (whose kernel uniformize a pseudo-real Riemann surface $S$) given by
$\theta(d_1)=a;\ \theta(\beta_1)=a^2b^{2^{n-3}};\ \theta(\beta_2)=b^{2^{n-3}-1};\ \theta(\beta_3)=b.$ 
\end{remark}

As consequence of the Theorem \ref{psrqan}, denoting the order of the largest quasi-abelian group of orientation-preserving automorphisms of a pseudo-real surface of genus $g$ by $M^+_{QA_n}(g)$, we observe $M^+_{QA_n}(g)\geq \frac{2}{3}g+2$, with $n\geq 4$.

%%%%%%%%%%%%%%%%%%%
\subsection{Actions of $K_{n}$ on pseudo-real Riemann surfaces}\label{CFinally}
Let $n\geq 4$ be an integer. First we observe that the group $K_n=\langle a, b\rangle$, with presentation as in (\ref{gKn}), has two index two subgroups. One is 
$QA_n=\langle b,a^{2}b^{2^{n-3}}\rangle$ and the other is $C_{2^{n-2}}\times C_4\cong K_{n}^{+}=\langle b^{-2}, a\rangle$ for $n\not=4$ ($C_{8}\times C_2\cong K_{4}^{+}=\langle a, b^4\rangle$ for $n=4$). Moreover, the groups $K_n$, $QA_n$ and $K_n^+$ have three involutions. 

\begin{lemma}\label{autkn}
The group $K_n$ has an automorphism that inverts each of the two generators $a$ and $b$.
\end{lemma}
\begin{proof}
The elements $a^{-1}$ and $b^{-1}$ satisfy the same defining relations for $K_n$ as the elements $a$ and $b$, because $a^{-1}$ and $b^{-1}$ have the same orders $8$ and $2^{n-2}$, and $b^{-1}a^{-1}b=(bab^{-1})^{-1}=(a^{\alpha})^{-1}=(a^{-1})^{\alpha}.$ Hence there exists an automorphism of $K_n$ taking $(a, b)$ to $(a^{-1}, b^{-1})$.
\end{proof}

For $L \in \{QA_{n}, K_{n}^{+}\}$, we set $\sigma_{ps}(K_n, L)$ as the minimal genus of a pseudo-real Riemann surface $S$ for which $K_{n}<{\rm Aut}(S)$ and $L=K_{n} \cap {\rm Aut}^{+}(S)$.

\begin{theo}\label{psKn} Let $n\geq 4$. Then
\begin{enumerate}
\item $\sigma_{ps}(K_n, QA_n)$ is equal to $3(2^{n-1}-1)$ for $n\geq 5$ and $17$ for $n=4$.
\item $\sigma_{ps}(K_n, K_n^+)=2^n+1$. 
\end{enumerate}
\end{theo} 

\begin{proof}
({\bf 1}) By Theorem \ref{psrqan}, we may a obtain pseudo-real Riemann surface $S$ of genus $g_0=3(2^{n-1}-1)$ for $n\not=4$ ($g_0=17$ for $n=4$) with ${\rm Aut}(S)=K_n$ and ${\rm Aut}^+(S)=QA_{n}$. 
  
 It remains to show that $g_{0}$ is minimal with that property. By the contrary, assume there is some pseudo-real Riemann surface $R$ of genus $g<g_{0}$ such that $K_n={\rm Aut}(R)$ and  $QA_n={\rm Aut}^{+}(R)$. In this case, there is an NEC group $\Delta$ of signature $(\gamma;-;[m_1,\cdots, m_r];\{-\})$, with $\gamma>0$ and $g=1+\frac{|K_n|}{2}|\Delta|^*$, where $|\Delta|^*=\gamma-2+\sum_{i=1}^r\left(1-\frac{1}{m_i}\right)$, such that ${\mathbb H}/\Delta=R/K_{n}$, and there is an epimorphism $\theta:\Delta \to K_{n}$ with kernel a torsion-free Fuchsian group uniformizing $R$. The condition $g<g_{0}$ asserts that  $|\Delta|^*<\frac{3}{2}-\frac{1}{2^{n-1}}$ for $n\not=4$ ($|\Delta|^*<1$ for $n=4$). Therefore, $\gamma \in \{1, 2,3\}$ for $n\not=4$ ($\gamma \in \{1, 2\}$ for $n=4$).  
 
 \noindent({\bf 2}) Let $n\geq 4$ be an integer. We consider an NEC group $\Delta$ with signature $(2;-;[2,2];\{-\})$ and presentation 
 $\Delta=\langle d_1,d_2,\beta_1,\beta_2:\ \beta_1^2=\beta_2^{2}=1, \beta_1\beta_2d_1^2d_2^2=1\rangle.$ Then quotient Klein surface uniformized by $\Delta$ is the orbifold whose underling surface is the connected sum of two real projective planes with 2 conical points. Let us consider the epimorphism $\theta:\Delta\to K_{n}$ given by
 $\theta(d_1)=b;\ \theta(d_2)=ab^{t};\ \theta(\beta_1)=\theta(\beta_2)=b^{2^{n-2}},$ where $t=2^{n-2}-1$ for $n$ even and $t=2^{n-3}-1$ for $n$ odd.
 
 Proceeding in a similar fashion as in the proof of Theorem \ref{tps0}, we may obtain a pseudo-real Riemann surface $S$ of genus $g_0=2^n+1$ with ${\rm Aut}(S) =K_{n}$ and ${\rm Aut}^{+}(S) =K_{n}^+$. 
It remains to show that $g_{0}$ is minimal with that property. By the contrary, assume there is some pseudo-real Riemann surface $R$ of genus $g<g_{0}$ such that $K_n={\rm Aut}(R)$ and  $K_n^+={\rm Aut}^{+}(R)$. In this case, there is an NEC group $\Delta$ of signature $(\gamma;-;[m_1,\cdots, m_r];\{-\})$, with $\gamma>0$ and $g=1+\frac{|K_n|}{2}|\Delta|^*$, such that ${\mathbb H}/\Delta=R/K_{n}$, and there is an epimorphism $\theta:\Delta \to K_{n}$ with kernel a torsion-free Fuchsian group uniformizing $R$. The condition $g<g_{0}$ asserts that  $|\Delta|^*<1$. Therefore, $\gamma \in \{1, 2\}$. 
 
 Finally, in both cases ({\bf 1}) and ({\bf 2}), by  Lemma  \ref{autkn} and similar  arguments as in the proof of Theorem \ref{mps+0} we obtain that $K_n$ can not acts on pseudo-real surfaces of genus $g$. \end{proof}
  
 The above provides the following. 

\begin{coro}\label{pkn}
If $n\geq 4$, then $\psi^*(K_n)=2^n+1$.
\end{coro}
%\medskip    
   
\subsection*{Some computations and questions} Using the automatized libraries of GAP \cite{GAP}, we find groups $G$ containing $QA_n$ as an index two subgroup, and also containing 
other different index two subgroups $H_i$ such that $G$, $QA_n$ and $H_i$ have the same number of involutions. For example, if $n=4$, we note that $G$ is isomorphic either to $K_4$ (Section \ref{Scaqan}) or to $G_1=\langle a,b,c:\ a^4=1, b^4=a^2, c^2=bab^{-1}=a^{-1}, ac=ca, cbc^{-1}=a^{-1}b^3\rangle$. By Theorem \ref{psKn},  $\sigma_{ps}(K_4, QA_4)=\psi^*(K_4)=17$. On the other hand, $G_1$ admits only two subgroups of index two ($QA_4$ and $C_2\times Q_8$). Using the capabilities of MAGMA \cite{Magma}, we get $\sigma_{ps}(G_1, QA_4)=21$, $\sigma_{ps}(G_1, C_2\times Q_8)=17$ and $\psi^*(G_1)=17$. So, we determine that $\psi^*(G_1)=\psi^*(K_4)$ and $\sigma_{ps}(K_4, QA_4)<\sigma_{ps}(G_1, QA_4)$.
  \medskip
  
   Finally, for $n\geq 5$ we add a few questions.
   
\noindent(1) What are the values of $\psi(G)$ and $\psi^*(G)$\,? \\
\noindent(2) Is there any closed relationship between $\psi^*(G)$ and $\psi^*(K_n)$\,? \\
\noindent(3) Is there any closed relationship between $\sigma_{ps}(G, QA_n)$ (when it is defined) and $\sigma_{ps}(K_n, QA_n)$\,? \\
\noindent(4) What is the value of $M^+_{QA_n}(g)$\,?

%%%%%%%%%%%%%%%%%%%%%%%%%%
%%%%%%%%%%%%%%%%%%%%%%%%%%
\section{Minimal genus actions of $QA_{n}$ on closed Klein surfaces} 
A compact hyperbolic Klein surface $X$ can be seen as a quotient $X=S/\langle \tau \rangle$, where $S$ 
is a closed Riemann surface of genus $g\geq 2$ (called the {\it algebraic genus} of $X$) and $\tau \in {\rm Aut}(S)$ is an anticonformal involution. The boundary of $X$ is provided by the projection of the fixed points of $\tau$.  The topological genus of  $X$ is  $\gamma=(g-k+1)/\eta$, where $\eta= 2$ if $S/\langle \tau \rangle$ is orientable and $\eta=1$ otherwise, and 
 $k$ is the number of connected components of the fixed points of $\tau$. In \cite[Theorem 2.5]{Bu}, it was observed that every finite group 
$G$ can be realized as a group of automorphisms of some closed Klein surface (i.e., non-orientable Klein surface with empty boundary). The minimal topological genus $\tilde{\sigma}(G)$ of these surfaces is called the {\em symmetric crosscap number} of $G$. It is know that $3$ cannot be the symmetric crosscap number of any group \cite{May0} and, also that if $H$ is a subgroup of $G$, then 
\begin{equation}\label{mcg}
\tilde{\sigma}(H)\leq \tilde{\sigma}(G).
\end{equation}
 
 \begin{theo}\label{ncqa}
If $n \geq 4$, then $\tilde{\sigma}(QA_{n})=2^{n-2}+2$ and this action of $QA_{n}$ is unique.
\end{theo}
\begin{proof}

Let  $X = \h/\Gamma$ be a closed Klein surface of topological genus $\gamma > 3$, such that $QA_{n} \leq {\rm Aut}(X)$. Then there exists an NEC group $\Delta$ such that $\Gamma$ is a normal subgroup of $\Delta$ with $QA_{n} = \Delta / \Gamma$ and (Riemann-Hurwitz formula) we have $\gamma-2 = |QA_{n}| |\Delta|^*$, where $|\Delta|^*$ is the reduced area of $\Delta$ (Section \ref{S2}). In this way, to compute $\tilde{\sigma}(QA_{n})$, we need to minimize the value of $|\Delta|^*$ among all those possible NEC groups $\Delta$. 
     
On the other hand, for $H \in\lbrace C_{2^{n-1}}, C_{2^{n-2}}\times C_2 \rbrace$ with $\tilde{\sigma}(H)>1$, we have \cite{Bu, ECM}: $$\tilde{\sigma}(C_{2^{n-1}})=2^{n-2}+1,\ \tilde{\sigma}(C_{2^{n-2}}\times C_2)=2^{n-2}.$$
 
The above, together to the inequality (\ref{mcg}), asserts that 
$$(*)\qquad 2^{n-2}+1 \leq \tilde{\sigma}(QA_{n}).$$ 

Since, the elements of order two of $QA_n$ are $y$, $x^{2^{n-2}}$ and $yx^{2^{n-2}}$ (Section \ref{Section 1}), we consider an NEC group $\Delta$ with signature $(0;+;[-];\{(-),(2)\})$ and presentation 
$\Delta=\langle e_1,\! e_2,\! c_{1j},\! c_{2j}: c_{1j}^2=c_{2j}^2=(c_{20}c_{21})^2=1,e_1e_2=e_ic_{i0}e_i^{-1}c_{i1}=1 \rangle$, and an epimorphism $\theta:\Delta\to QA_n$ given by
$\theta(e_1)=x^{-1};\theta(e_2)=x; \theta(c_{1j})=x^{2^{n-2}}; \theta(c_{20})=y; \theta(c_{21})=yx^{2^{n-2}}.$

Observe that $\theta(c_{21}c_{10})=y$, thereby $\theta(\Delta^+)=QA_n$ and there exists orientation-reversing elements in ${\rm ker}(\theta)$. As $|\Delta|^*=1/4$, these implies that $\tilde{\sigma}(QA_{n})\leq 2^{n-2}+2$. Only we need to prove that this bound cannot be lowered. The inequality $(*)$ asserts that the unique value to be discarded is $2^{n-2}+1$.

As the group $QA_n$ can be generated by two elements of order 2 and $2^{n-1}$, from \cite[Table 1]{ECM} we see that there is not admissible signatures such that the genus is $2^{n-2}+1$. Therefore, $\tilde{\sigma}(QA_{n})=2^{n-2}+2$.

In order to prove the uniqueness of the $QA_{n}$-action we proceed as follows.
As noted above, the minimal action is provided by 
an NEC group $\Delta$ with signature of the form $(0;+;[-];\{(-), (2)\})$. An epimorphism $\theta:\!\Delta\to QA_{n}$ must satisfy that $\theta(e_{1})$, $\theta(e_{2})\! \in\! \lbrace x^{l}, yx^{l}, l-{\rm odd} \rbrace\subset \Delta^+;$  $\theta(c_{is_i})\!\in\! \Delta\setminus \Delta^{+};$ $\theta(e_{i})\theta(c_{i0}) \theta(e_{i})^{-1}\theta(c_{i1})=1$, $\theta(e_{2})=\theta(e_1)^{-1}.$
By post-composing $\theta$ by automorphisms of $QA_{n}$, we may obtain that the unique epimorphism is given by
$\theta(e_1)=x^{-1};\theta(e_2)=x; \theta(c_{10})=\theta(c_{11})=x^{2^{n-2}}; \theta(c_{20})=y; \theta(c_{21})=yx^{2^{n-2}}.$ 
\end{proof} 
  
%%%%%%%%%%%%%%%%%
%%%%%%%%%%%%%%%%%%
\medskip
\noindent{\bf Acknowledgements}\\
The results of this article are mostly based on the second author's Ph.D thesis.

\end{document}